\documentclass{amsart}
\usepackage{latexsym, amssymb, amsmath, mathrsfs, amsfonts}

\newtheorem{thm}{Theorem}[section]
\newtheorem{lemm}[thm]{Lemma}

\newtheorem{prop}[thm]{Proposition}
\theoremstyle{remark}
\newtheorem{rmk}[thm]{Remark}
\theoremstyle{definition}
\newtheorem{defi}[thm]{Definition}

\title{A gap theorem for positive Einstein metrics\\ on the four-sphere}

\author{Kazuo Akutagawa${}^*$}
\email{akutagawa@math.titech.ac.jp}
\address{Department of Mathematics, Tokyo Institute of Technology, 
Tokyo 152-8551, Japan}

\author{Hisaaki Endo${}^{**}$}
\email{endo@math.titech.ac.jp}
\address{Department of Mathematics, Tokyo Institute of Technology, 
Tokyo 152-8551, Japan}

\author{Harish Seshadri}
\email{harish@math.iisc.ernet.in}
\address{Mathematics Department, Indian Institute of Science, 560012 Bangalore, India}

\thanks{${}^*$\ 
supported in part by the Grants-in-Aid for Scientific Research (B), 
Japan Society for the Promotion of Science, No.~24340008.} 

\thanks{${}^{**}$\ 
supported in part by the Grants-in-Aid for Scientific Research (C), 
Japan Society for the Promotion of Science, No.~16K05142.} 

\date{January, 2018;\ \ February, 2018 (revised version).}
\begin{document} 
\maketitle
\markboth{A gap theorem for positive Einstein metrics on the four-sphere}
{Kazuo Akutagawa, Hisaaki Endo and Harish Seshadri}

\begin{abstract}
 We show that there exists a universal positive constant $\varepsilon_0 > 0$ with the following property:  Let $g$ be a positive Einstein metric on $S^4$.    If the  Yamabe constant of the conformal class $[g]$ satisfies
 $$ Y(S^4, [g]) >\frac{1}{\sqrt{3}} Y(S^4, [g_{\mathbb S}])  - \varepsilon_0\,, $$ 
where $g_{\mathbb S}$ denotes the standard round metric on $S^4$,  then,  up to rescaling, $g$ is isometric to $g_{\mathbb S}$. 
This is an extension of Gursky's gap theorem for positive Einstein metrics on the four-sphere.   
\end{abstract}
\maketitle
  
\section{Introduction and main results} 

A smooth Riemannian metric $g$ is said to be {\it Einstein} if its Ricci tensor ${\rm Ric}_g$ is a constant multiple $\lambda$ of $g$\,: 
$$ 
{\rm Ric}_g = \lambda\,g\,.
$$ 
When such a metric exists, it is natural to ask whether it is unique. 
However in dimension $n \geq 5$, there exist many examples of closed $n$-manifolds 
each of which has infinitely many non-homothetic Einstein metrics (cf.\,\cite{Besse}). 
In fact, there exists infinitely many non-homothetic
 Einstein metrics of positive sacalar curvature 
({\it positive Einstein} for brevity) on $S^n$ when $5 \le n \le 9$ \cite{Bohm} (cf.\,\cite{Jensen}, \cite{B-K}). 
There are no non-existence or uniqueness results known when $n \geq 5$. 

When $n = 4$, there are necessary topological conditions for a  closed $4$-manifold $M$ to admit an Einstein metric \cite{Thorpe}, \cite{Hitchin-1}, \cite{LeBrun-2}. 
Uniqueness is known in some special cases: when $M$ is a smooth compact quotient of real hyperbolic $4$-space ({\rm resp.} complex-hyperbolic $4$-space), 
the standard negative Einstein metric is the unique Einstein metric (up to rescaling and isometry) \cite{BCG} ({\rm resp.}\,\cite{LeBrun-1}). 
In the positive case, there are  some partial rigidity results on the $4$-sphere $S^4$ and the complex projective plane $\mathbb{CP}^2$ \cite{GL}, \cite{G}, \cite{Y}. 
When $M = S^4$, the standard round metric $g_{\mathbb{S}}$ of constant curvature $1$ is, to date, the only known Einstein metric (up to rescaling and isometry). 
In this connection we have the following gap theorem due to M.\,Gursky (see \cite{ABKS} for the significance of the constant $\frac{1}{\sqrt{3}} Y(S^4, [g_{\mathbb S}])$)\,:

\begin{thm}[M.\,Gursky\,\cite{G}]\label{Gursky} 
Let $g$ be a positive Einstein metric on $S^4$. 
If its Yamabe constant $Y(S^4, [g])$ satisfies the following inequality 
$$ 
Y(S^4, [g]) \geq \frac{1}{\sqrt{3}} Y(S^4, [g_{\mathbb S}]) 
$$ 
then, up to rescaling, $g$ is isometric to  $g_{\mathbb{S}}$. 
Here, $[g]$ denotes the conformal class of $g$. 
\end{thm} 

Note that $Y(S^4, [h]) \leq Y(S^4, [g_{\mathbb{S}}]) = 8\sqrt{6}\pi$ for any Riemannian metric $h$ and that $Y(S^4, [g]) = R_g \sqrt{V_g}$ for any Einstein metric $g$, 
where $R_g$ and $V_g = {\rm Vol}(S^4, g)$ denote respectively the scalar curvature of $g$ and the volume of $(S^4, g)$.  

Our main result in this paper is an extension of Theorem\,\ref{Gursky}:

\begin{thm}\label{MainThm1} 
There exists a universal positive constant $\varepsilon_0 > 0$ with the following property\,$:$ 
If $g$ is a positive Einstein metric on $S^4$ 
with Yamabe constant 
$$ 
Y(S^4, [g]) >\frac{1}{\sqrt{3}} Y(S^4, [g_{\mathbb S}])  - \varepsilon_0, 
$$ 
then, up to rescaling, $g$ is isometric to $g_{\mathbb{S}}$. 
\end{thm} 

This result can be restated in terms of the {\it Weyl constant} of $[g]$ (cf.\,\cite{ABKS}). 
Indeed, the Chern-Gauss-Bonnet theorem (see Remark\,\ref{ALE}-(1)) implies that the lower bound on the Yamabe constant 
is equivalent to the following upper bound on the Weyl constant\,:  $\int_M |W_g|^2 d  \mu_g < \frac{32}{3} \pi^2 + \widetilde{\varepsilon}_0$, 
where $\widetilde{\varepsilon}_0 := \frac{\varepsilon_0}{24}(16\sqrt{2}\pi - \varepsilon_0) > 0$. 
 
More generally, we obtain the following (note that $8\sqrt{2}\pi = \frac{1}{\sqrt{3}} Y(S^4, [g_{\mathbb S}])$):

\begin{thm}\label{MainThm2} 
For $c > 0$, let $\mathcal{E}_{\geq c}(S^4)$ 
denote the space of all unit-volume positive Einstein metrics $g$ on $S^4$ with $c \leq Y(S^4, [g]) < 8\sqrt{2}\pi$. 
Then the number of connected components of the moduli space $\mathcal{E}_{\geq c}(S^4)/{\rm Diff}(S^4)$  is finite. 
In particular,  $\{ Y(S^4, [g]) \in [c, 8\sqrt{2}\pi)\ |\ g \in \mathcal{E}_{\geq c}(S^4) \}$ is a finite set $($possible empty$)$.  
\end{thm} 
Here $ \mathcal{M}_1(S^4)/{\rm Diff}(S^4)$ has the $C^\infty$-topology  and  $\mathcal{E}_{\geq c}(S^4)/{\rm Diff}(S^4)$  is endowed with the subspace topology. 

These theorems follow from the following crucial result:

\begin{prop}\label{MainProp} 
Let $\{g_i\}$ be a sequence in $  \mathcal{E}_{\geq c}(S^4)$ for some positive constant $c > 0$. 
Then there exists a subsequence $\{j\} \subset \{i\}$, $\{\phi_j\} \subset {\rm Diff}(S^4)$ 
and a unit-volume positive Einstein metric $g_{\infty}$ on $S^4$ 
such that $\phi_j^*g_j$ converges to $g_{\infty}$ with respect to the $C^{\infty}$-topology on $\mathcal{M}_1(S^4)$. 
\end{prop}

\noindent {\bf Remark:} Theorem\,D of  \cite{Anderson} states that the same conclusion as the one in Proposition\,\ref{MainProp} holds 
for any sequence $\{g_i\} \subset \mathcal{E}_{\geq c}(M)$ on any closed $4$-manifold $M$ with $1 \leq \chi(M) \leq 3$, 
where $\chi(M)$ denotes the Euler characteristic of $M$. 
Unfortunately, the proof appears to be incorrect. Specifically, Theorem D is based on Lemma 6.3, which asserts that  
a Ricci-flat ALE 4-space $X$ with $\chi(X)=1$ is necessarily isometric to the Euclidean $4$-space $({\mathbb R}^4, g_{\mathbb{E}})$.
This is not true: the Ricci-flat ALE 4-space $X_1$ constructed by Eguchi-Hanson \cite{EH} has a free, isometric ${\mathbb Z}_2$-action 
whose quotient  $X_2 = X_1/{\mathbb Z}_2$ is a Ricci-flat ALE $4$-space with $\chi(X_2)=1$. 
Note that $X_2$ is nonorientable. 
Even if we assume that $X$ is orientable in Lemma\,6.3, the topological argument in the proof still contains some gaps. 
Proposition\,3.10 of \cite{Anderson-GAFA} corrects a minor inaccuracy of Lemma\,6.3. 
However, the proof also contains some gaps in the topological argument (see Remark\,\ref{Counter} in $\S$\,4 for details).

\vspace{2mm}

Gursky's proof of Theorem\,\ref{Gursky}  involves a sophisticated Bochner technique, 
a modified scalar curvature and a conformal rescaling argument.  The proof of Proposition\,1.4 
is based on topological results about $S^3$-quotients embedded in $S^4$ 
and the convergence theory of Einstein metrics in four-dimensions.  
Given this proposition, we invoke Gursky's result to prove Theorems\,\ref{MainThm1} and\,\ref{MainThm2}.

In $\S$\,2, we recall some background material and prove Theorems\,\ref{MainThm1} and\,\ref{MainThm2}, assuming Proposition\,1.4.
In $\S$\,3, we review two key results needed for the proof of Proposition\,1.4. 
Finally, in $\S$\,4, we prove Proposition\,1.4.  
\quad \\ 
\quad \\ 
\noindent
{\bf Acknowledgements.} 
The authors would like to thank Anda Degeratu and Rafe Mazzeo for valuable discussions on the eta invariant, 
and Shouhei Honda for helpful discussions on convergence results of Riemannian manifolds with bounded Ricci curvature. 
They would also like to thank Matthew Gursky and Claude LeBrun for useful advices, 
and Gilles Carron for crucial comments. \\

\section{Preliminaries and proofs of Theorems 1.2 and 1.3} 

We first review the definitions of Yamabe constants and Yamabe metrics. 
Let $M^n$ be a closed $n$-manifold with $n \geq 3$. 
It is well known that a Riemannian metric on $M$ is Einstein if and only if it is a critical point of 
the normalized Einstein-Hilbert functional $I$ on the space $\mathcal{M}(M)$ of all Riemannian metrics on $M$
$$ 
I : \mathcal{M}(M) \rightarrow \mathbb{R},\quad g \mapsto I(g) := \frac{\int_MR_gd\mu_g}{{\rm Vol}(M, g)^{(n-2)/n}}, 
$$ 
where $d\mu_g$ denotes the volume form of $g$. 
The restriction of $I$ to any conformal class $[g] := \{ e^{2f}\,g\ |\ f \in C^{\infty}(M) \}$ is always bounded from below. 
Hence, we can consider the following conformal invariant  
$$
Y(M, [g]) := \inf_{\widetilde{g} \in [g]}I(\widetilde{g}), 
$$
which is called the {\it Yamabe constant} of $(M, [g])$. 
A remarkable theorem of H.\,Yamabe, N.\,Trudinger, T.\,Aubin and R.\,Schoen asserts that 
each conformal class $[g]$ contains metrics $\check{g}$, called {\it Yambe metrics}, 
which realize the minimum (cf.\,\cite{LP}, \cite{Sc-1}) 
$$ 
Y(M, [g]) = I(\check{g}).
$$
These metrics must have constant scalar curvature 
$$ 
R_{\check{g}} = Y(M, [g])\cdot V_{\check{g}}^{-2/n}, 
$$
where $V_{\check{g}} = {\rm Vol}(M, \check{g})$. 
Aubin proved that 
$$
Y(M^n, C) \leq Y(S^n, [g_{\mathbb{S}}]) = n(n-1) V_{g_{\mathbb{S}}}^{2/n}
$$
for any conformal class $C$ on $M$. 
Obata's Theorem\,\cite{Obata} implies that {\it any Einstein metric is a Yamabe metric}. 
When $n = 4$, 
$$
Y(M^4, [g]) = R_{\widehat{g}} \sqrt{V_{\widehat{g}}} \leq Y(S^4, [g_{\mathbb{S}}]) = 8\sqrt{6}\pi
$$  
for any Einstein metric $\widehat{g} \in [g]$. 

Assuming Proposition\,1.4, we can now prove Theorem\,\ref{MainThm1}. 

\begin{proof}[Proof of Theorem\,\ref{MainThm1}] 
Suppose that there exists a sequence $\{g_i\}$ of unit-volume Einstein metrics on $S^4$ satisfying 
$$ 
Y(S^4, [g_i]) = R_{g_i} < 8\sqrt{2}\pi\ \ ({\rm for}\ \ \forall i),\quad Y(S^4, [g_i]) = R_{g_i} \nearrow 8\sqrt{2}\pi\ \ ({\rm as}\ \ i \to \infty). 
$$ 
By Proposition\,1.4, 
there exists a subsequence $\{j\} \subset \{i\}$, a sequence $\{\phi_j\} \subset {\rm Diff}(S^4)$ 
and a unit-volume positive Einstein metric $g_{\infty}$ on $S^4$ 
such that $\phi_j^*g_j$ converges to $g_{\infty}$ with respect to the $C^{\infty}$-topology on $S^4$. 
Then, we get 
\begin{equation}\label{lll}
Y(S^4, [g_{\infty}]) = R_{g_{\infty}} = 8\sqrt{2}\pi. 
\end{equation}

On the other hand, Theorem\,\ref{Gursky} implies that $(S^4, g_{\infty})$ is isometric to $(S^4, g_{\mathbb{S}})$. 
Hence, 
$$ 
Y(S^4, [g_{\infty}]) = Y(S^4, [g_{\mathbb{S}}]) = 8\sqrt{6}\pi. 
$$
This contradicts (\ref{lll}). 
Therefore, there exists a positive constant $\varepsilon_0 > 0$ such that 
any unit-volume positive Einstein metric $g$ on $S^4$ satisfying 
$$ 
Y(S^4, [g]) > 8\sqrt{2}\pi - \varepsilon_0
$$ 
is isometric to $g_{\mathbb{S}}$. 
\end{proof}

\begin{proof}[Proof of Theorem\,\ref{MainThm2}] 
By the result of N. Koiso\,\cite[Theorem\,3.1]{Koiso-2} and \cite[Corollary\,12.52]{Besse} (cf.\,\cite[Theorem\,7.1]{Ebin}, \cite[Theorem\,2.2]{Koiso-1}), 
we first remark that, for each $g \in \mathcal{E}_{\geq c}(S^4)$, the premoduli space $\mathcal{E}_{\geq c}(S^4)$ around $g$ 
is a real analytic subset of a finite dimensional real analytic submanifold in $\mathcal{M}_1(S^4) := \{ g \in \mathcal{M}(S^4)\ |\ V_g = 1 \}$, 
and the moduli space $\mathcal{E}_{\geq c}(S^4)/{\rm Diff}(S^4)$ is arcwise connected. 
Moreover, the Yamabe constant $Y(S^4, [\bullet])$ is a locally constant function and
and it takes (at most) countably many values on $\mathcal{E}_{\geq c}(S^4)/{\rm Diff}(S^4)$. 

Suppose that there exist infinitely many connected components of the moduli space $\mathcal{E}_{\geq c}(S^4)/{\rm Diff}(S^4)$ 
(see \cite[Chapters\,4 and 12]{Besse} for the topology on it). 
Then, there exists a sequence $\{g_i\}$ in $\mathcal{E}_{\geq c}(S^4)$ such that 
the equivalence classes of any two $g_{i_1}$ and $g_{i_2}$ are contained in different connected components of $\mathcal{E}_{\geq c}(S^4)/{\rm Diff}(S^4)$ if $i_1 \ne i_2$. 
Similar to the proof of Theorem\,\ref{MainThm1}, 
there exists a subsequence $\{j\} \subset \{i\}$, a sequence $\{\phi_j\} \subset {\rm Diff}(S^4)$ 
and a unit-volume positive Einstein metric $g_{\infty}$ on $S^4$ 
such that $\phi_j^*g_j$ converges to $g_{\infty}$ with respect to the $C^{\infty}$-topology on $S^4$. 
We note here that the topology of the moduli space is induced from the one of the space $\mathcal{M}_1(S^4)$. 
Then, there exists a large positive integer $j_0$ such that the set $\{\phi_j^*g_j\}_{j \geq j_0}$ is contained in a connected component. 
This contradicts the choice of $\{g_i\}$. 
Hence, the number of connected components of the moduli space $\mathcal{E}_{\geq c}(S^4)/{\rm Diff}(S^4)$ is finite (possibly zero). 
In particular, the set of $\{ Y(S^4, [g]) \in [c, 8\sqrt{2}\pi)\ |\ g \in \mathcal{E}_{\geq c}(S^4) \}$ is a finite set $($possibly empty$)$.  
\end{proof}

\section{A review of two key results}

\subsection{An embedding theorem:}  It will be  necessary to know which quotients $S^3/\Gamma$ of $S^3$  embed smoothly  in $S^4$. 
The theorem below gives a complete answer, 
which is one of the two key results for the proof of Proposition\,1.4. 

\begin{thm}[Crisp-Hillman\,\cite{CH}]\label{Key-2} 
Let $\Gamma \subset SO(4)$ be a finite subgroup such that $S^3/\Gamma$ is a smooth quotient of $S^3$. 
If $S^3/\Gamma$ can be smoothly embedded in $S^4$, 
then either $\Gamma = \{1\}$ or $\Gamma = Q_8$. 
Here, $Q_8 = \{\pm 1, \pm i, \pm j, \pm k\}$ denotes the quaternion group. 
\end{thm}

\subsection{Convergence of Einstein metrics:}

We first review the definition of the energy of metrics on $4$-manifolds.

\begin{defi} $(1)$ \ For a closed Riemannian $4$-manifold $(M, g)$, 
the {\it energy} $\mathscr{E}(g)$ of $g$ (or $(M, g)$) is defined by 
$$
\mathscr{E}(g) := \frac{1}{8\pi^2}\int_M|\mathscr{R}_g|^2d\mu_g,  
$$ 
where $\mathscr{R}_g = (R^i_{\ jk\ell})$ denotes the curvature tensor of $g$ and $|\mathscr{R}_g|^2 = \frac{1}{4}R^i_{\ jk\ell}R_i^{\ jk\ell}$. \vspace{2mm}

$(2)$ \ If $(X,h)$ is an  {\it asymptotically locally Euclidean} $4$-manifold of order $\tau > 0$ (ALE $4$-space for brevity, cf.\,\cite{BKN}), 
the energy $\mathscr{E}(h)$ of $h$ (or $(X, h)$) is again defined by 
$$
\mathscr{E}(h) := \frac{1}{8\pi^2}\int_X|\mathscr{R}_h|^2d\mu_h < \infty.    
$$ 
\end{defi}
\vspace{2mm}

\begin{rmk}\label{ALE} 
$(1)$ \ By the Chern-Gauss-Bonnet formula, $\mathscr{E}(g) = \chi(M)$ for any Einstein metric $g$ on a closed $4$-manifold $M$. 
Indeed, 
\begin{align*}
\mathscr{E}(g) &= \frac{1}{8\pi^2}\int_M|\mathscr{R}_g|^2d\mu_g = \frac{1}{8\pi^2}\int_M(|W_g|^2 + \frac{1}{24}R_g^2 + \frac{1}{2}|E_g|^2)d\mu_g\\ 
&= \frac{1}{8\pi^2}\int_M(|W_g|^2 + \frac{1}{24}R_g^2 - \frac{1}{2}|E_g|^2)d\mu_g = \chi(M),
\end{align*}
where $W_g = (W^i_{\ jk\ell})$ and $E_g = (E_{ij})$ denote respectively the Weyl tensor and the trace-free Ricci tensor ${\rm Ric}_g - \frac{R_g}{4}g$ of $g$, 
and $|W_g|^2 = \frac{1}{4}W^i_{\ jk\ell}W_i^{\ jk\ell}$. 
In particular, $\mathscr{E}(g) = 2$ if $M = S^4$.\\ 
$(2)$ \ The Chern-Gauss-Bonnet formula  for $4$-manifolds with boundary implies 
the following (cf.\,\cite[formula\,(7)]{Hitchin-2}): 
any Ricci-flat ALE $4$-space $(X, h)$ with end $S^3/\Gamma$ satisfies 
$$ 
\chi(X) = \mathscr{E}(h) + \frac{1}{|\Gamma|}, 
$$ 
where $\Gamma$ is a finite subgroup of $O(4)$ acting freely on $\mathbb{R}^4 - \{0\}$ and $|\Gamma|$ is the order of $\Gamma$. 
If $\chi(X) = 1$, we get, in particular, the following: 
$$ 
\mathscr{E}(h) = 1 - \frac{1}{|\Gamma|}. 
$$ 
$(3)$ \ Bando-Kasue-Nakajima\,\cite{BKN} proved that any Ricci-flat ALE $4$-space $(X, h)$ 
is an ALE $4$-space of order of $4$. 
Moreover, when $(X, h)$ is {\it asymptotically flat} (AF for brevity, cf.\,\cite{Bartnik}), that is, $\Gamma = \{1\}$, 
this combined with a result of  R. Bartnik\,\cite[Theorem\,4.3]{Bartnik} implies that the mass of $(X, h)$ is zero. 
The Positive Mass Theorem\,\cite[Theorem\,4.3]{Sc-1} for AF manifolds then implies that $(X, h)$ is isometric to $(\mathbb{R}^4, g_{\mathbb{E}})$. 
Note that $\mathscr{E}(h) = \mathscr{E}(g_{\mathbb{E}}) = 0$. 
\end{rmk}

Recall again that any Einstein metric $g$ on a closed $4$-manifold $M$ satisfies that $Y(M, [g]) = R_g\sqrt{V_g}$. 
Moreover, if $g$ is a unit-volume Einstein metric with $Y(M, [g]) \geq c\ (c > 0)$, then ${\rm Ric}_g \geq \frac{c}{4}g$.   
Hence, Myers' diameter estimate gives  
$$ 
{\rm diam}(M, g) \leq \frac{2\sqrt{3}\pi}{\sqrt{c}}. 
$$ 

Using this fact and Remark\,3.3-(1), we can now state a modified version of the convergence theorem for Einstein metrics 
due to M. Anderson\,\cite{Anderson}, H.  Nakajima\,\cite{Nakajima-1} and Bando-Kasue-Nakajima\,\cite{BKN}, 
which is the other of the two key results for the proof of Proposition\,1.4. 

\begin{thm}\label{Key-1}
Let $M$ and $\{g_i\}$ be respectively a closed $4$-manifold 
and a sequence of unit-volume positive Einstein metrics on $M$ with $Y(M, [g]) \geq c$ for a fixed $c > 0$. 
Then, there exist a subsequence $\{j\} \subset \{i\}$ and a compact Einstein $4$-orbifold $(M_{\infty}, g_{\infty})$ 
with a finite singular points $\mathcal{S} = \{p_1, p_2, \cdots, p_{\ell}\} \subset M_{\infty}$ $($possibly empty$)$ 
and an orbifold structure group $\Gamma_a \subset O(4)$ around $p_a$ 
for which the following assertions hold\,$:$\\ 
\quad $(1)$\ \ $(M, g_j)$ converges to $(M_{\infty}, g_{\infty})$ in the Gromov-Hausdorff distance. \\ 
\quad $(2)$\ \ There exists a smooth embedding $\phi_j : M_{\infty} - \mathcal{S} \rightarrow M$ for each $j$ such that 
$\phi_j^*g_j$ converges to $g_{\infty}$ in the $C^{\infty}$-topology on $M_{\infty} - \mathcal{S}$. 
If $\mathcal{S}$ is empty, then each $\phi_j$ is a diffeomorphism from $M_{\infty}$ onto $M$. \\ 
\quad $(3)$\ \ For each $p_a \in \mathcal{S}$ and $j$, there exists $p_{a, j} \in M$ and a positive number $r_j$ such that \\
\quad $(3.1)$\ \ $B_{\delta}(p_{a, j}; g_j)$ converges to $B_{\delta}(p_a;g_{\infty})$ in the pointed Gromov-Hausdorff distance for all $\delta > 0$, 
where $B_{\delta}(p_{a, j}; g_j)$ denotes the geodesic ball of radius $\delta > 0$ centered at $p_{a, j}$ with respect to $g_j$. \\ 
\quad $(3.2)$\ \ $\lim_{j \to \infty}r_j = 0$. \\ 
\quad $(3.3)$\ \ $((M, r_j^{-2}g_j), p_{a, j})$ converges to $((X_a, h_a), x_{a, \infty})$ in the pointed Gromov-Hausdorff distance, 
where $(X_a, h_a)$ is a complete, non-compact, Ricci-flat, non-flat ALE $4$-space of order $4$ with 
$$ 
0 < \int_{X_a}|\mathscr{R}_{h_a}|^2d\mu_{h_a} < \infty, 
$$ 
and $x_{a, \infty} \in X_a$. \\
\quad $(3.4)$\ \ There exists smooth embeddings $\Phi_j : X_a \rightarrow M$ such that 
$\Phi_j^*(r_j^{-2}g_j)$ converges to $h_a$ in the $C^{\infty}$-topology on $X_a$. \\ 
\quad $(4)$\ \ It holds that 
$$ 
\lim_{j \to \infty}\int_M|\mathscr{R}_{g_j}|^2d\mu_{g_j} \geq \int_{M_{\infty}}|\mathscr{R}_{g_{\infty}}|^2d\mu_{g_{\infty}} + \sum_a\int_{X_a}|\mathscr{R}_{h_a}|^2d\mu_{h_a}. 
$$ 
\end{thm}

\begin{rmk}\label{Tree} 
Since $S^3/\Gamma_a$ is smoothly embedded in $M_{\infty}$ around $p_a$ for each $a$, 
it is also smoothly embedded in $M$ and it separates $M$ into two components $V_a, W_a$, 
which are compact $4$-manifolds with boundary. 
More precisely, $M = V_a \cup W_a,\ S^3/\Gamma_a = \partial V_a = \partial W_a = V_a \cap W_a$. 
Here, we choose $V_a$ satisfying $V_a \subset M_{\infty}$. 
The infinity $X_a(\infty) \cong S^3/\widetilde{\Gamma}_a$ of $X_a$ is also smoothly embedded in $M$. 
By the existence of intermediate Ricci-flat ALE $4$-orbifolds in the bubbling tree arising from each singular point $p_a$, 
$\Gamma_a \ne \widetilde{\Gamma}_a$ in general (cf.\,\cite{Bando}, \cite{Nakajima-3}).
\end{rmk}

\section{Proof of Proposition\,1.4}

Let $\{g_i\}$ be a sequence of positive Einstein metrics on $S^4$ with $\{g_i\} \subset \mathcal{E}_{\geq c}(S^4)$ for some $c > 0$. 
We apply Theorem\,\ref{Key-1}, with $M=S^4$, for  the sequence $\{g_i\}$. 
Then, in order to prove Proposition\,1.4, by Theorem\,\ref{Key-1}-(2), it is enough to show that 
the singular set $\mathcal{S}$ is empty. 

From now on, \underline{suppose that $\mathcal{S} \ne \emptyset$}, that is, $\ell \geq 1$. 
By a similar reason to Remark\,\ref{Tree}, the case of $\Gamma_a = \{1\}$ for some $a\ (1 \leq a \leq \ell)$ may be possible logically 
for a general closed $4$-manifold, particularly $\Gamma_a \ = \{1\}$ for all $a = 1, 2, \cdots, \ell$.  
However, at least in the case of $M = S^4$, the following holds. 
We use the notation of Theorem\,\ref{Key-1} and Remark\,\ref{Tree}. 

\begin{lemm}\label{Subkey}
$\Gamma_{a_0} \ne \{1\}$ for some $a_0\ (1 \leq a_0 \leq \ell)$. 
\end{lemm} 

\begin{proof} 
Suppose that $\Gamma_a = \{1\}$ for all $a = 1, 2, \cdots, \ell$.  
As mentioned in Remark\,\ref{Tree}, a smooth embedded $S^3$ around $p_1 \in S^4$ separates $S^4$ into 
two components $V_1, W_1$ of compact $4$-manifolds with boundary satisfying 
$$ 
S^4 = V_1 \cup W_1,\quad S^3 = \partial V_1 = \partial W_1 = V_1 \cap W_1,\quad V_1 \subset S^4_{\infty}. 
$$ 
By the Mayer-Vietoris exact sequence of homology groups for $(S^4; V_1, W_1)$, 
one can get 
$$ 
H_i(V_1; \mathbb{R}) = H_i(W_1; \mathbb{R}) = 0\qquad {\rm for}\ \ i = 1, 2, 3, 
$$ 
and hence $\chi(V_1) = \chi(W_1) = 1$. 
Let $S^4_1 := V_1 \cup_{S^3} \overline{D^4}$ be a closed smooth $4$-manifold obtained by gluing along $S^3 = \partial V_1 = \partial\overline{D^4}$, 
where $\overline{D^4}$ denotes the closed $4$-ball in $\mathbb{R}^4$. 
Note that $\chi(S^4_1) = 2$. 
Similar to the above, a smooth embedded $S^3$ around $p_2 \in S^4_1$ separates $S^4_1$ into two components $V'_2, W'_2$. 
Then, the closed smooth $4$-manifold $S^4_2 := V'_2 \cup_{S^3} \overline{D^4}$ also satisfies that $\chi(S^4_2) = 2$. 
Repeating a similar procedure up to $a = \ell$, we get finally a closed smooth $4$-manifold $S^4_{\ell} := V'_{\ell} \cup_{S^3} \overline{D^4}$ with $\chi(S^4_{\ell}) = 2$. 
By construction, $S^4_{\ell}$ is homeomorphic to $S^4_{\infty}$ which implies that $\chi(S^4_{\infty}) = 2$. 

By the removable singularities theorem for Einstein metrics \cite[Theorem\,5.1]{BKN}, 
we note that $(S^4_{\infty}, g_{\infty})$ is a closed {\it smooth} Einstein $4$-manifold. 
Combining this with $\chi(S^4_{\infty}) = 2$, we get that $\mathscr{E}(g_{\infty}) = 2$. 
However, each Ricci-flat ALE $4$-space $(X_a, h_a)$ bubbling out from $p_a$ has a positive energy $\mathscr{E}(h_a) > 0$. 
This, combined with Theorem\,\ref{Key-1}-(4) leads to a contradiction: 
$$ 
2 = \lim_{j \to \infty}\mathscr{E}(g_j) \geq \mathscr{E}(g_{\infty}) + \sum_a\mathscr{E}(h_a) > 2. 
$$ 
Therefore, $\Gamma_{a_0} \ne \{1\}$ for some $a_0\ (1 \leq a_0 \leq \ell)$. 
\end{proof} 

We can now prove Proposition\,1.4. 

\begin{proof}[Proof of Proposition\,1.4] 
For simplicity, we assume that $a_0 = 1$. 
It then follows from Theorem\,\ref{Key-2} that $\Gamma_1 = Q_8$. 
By Remark\,\ref{ALE}-(3), we also obtain that $\widetilde{\Gamma}_1 = Q_8$. 
Even if $\widetilde{\Gamma}_a = Q_8$ for some $a$, a similar Mayer-Vietoris argument to that in the proof of Lemma\,\ref{Subkey} still holds, 
and so $\chi(X_1) = 1$. 
It then follows from Remark\,\ref{ALE}-(2) that 
$$ 
\mathscr{E}(h_1) = \chi(X_1) - \frac{1}{|Q_8|} = 1- \frac{1}{8} = \frac{7}{8}. 
$$ 

By the signature theorem for compact $4$-orbifolds (cf.\,\cite[(4.5)]{Nakajima-2}) and the calculation 
of eta invariant $\eta_S(S^3/\Gamma)$ for the signature operator \cite[Section\,3]{Hitchin-2}, 
the compact Einstein $4$-orbifold $(S^4_{\infty}, g_{\infty})$ satisfies that 
$$ 
\tau(S^4_{\infty}) = \frac{1}{12\pi^2}\int_{S^4_{\infty}}\Big{(} |W_{g_{\infty}}^+|^2 - |W_{g_{\infty}}^-|^2 \Big{)}d\mu_{g_{\infty}} 
- \sum_{a=1}^{\ell}\eta_S(S^3/\Gamma_a),\quad \eta_S(S^3/Q_8) = \frac{3}{4}, 
$$ 
where $\tau(S^4_{\infty})$ denotes the signature of $S^4_{\infty}$. 
Since $H_2(S^4_{\infty}; \mathbb{R}) = 0$, we have that $H^2(S^4_{\infty}; \mathbb{R}) = 0$, and so $\tau(S^4_{\infty}) = 0$. 
Combining that $R_{g_{\infty}} \geq c > 0$ and $\eta_S(S^3/\Gamma_a) \geq 0$ with the above and Theorem\,\ref{Key-1}-(4), we then obtain that 
$$ 
\frac{9}{8} = 2 - \frac{7}{8} \geq \mathcal{E}(g_{\infty}) = \frac{1}{8\pi^2}\int_{S^4_{\infty}}\Big{(} |W_{g_{\infty}}|^2 + \frac{R_{g_{\infty}}^2}{24} \Big{)}d\mu_{g_{\infty}} 
> \frac{1}{8\pi^2}\int_{S^4_{\infty}}|W_{g_{\infty}}|^2d\mu_{g_{\infty}} 
$$ 
$$ 
\qquad \quad \geq \frac{1}{8\pi^2}\int_{S^4_{\infty}}|W_{g_{\infty}}^+|^2d\mu_{g_{\infty}} \geq \frac{3}{2}\Big{(} \frac{3}{4} 
+ \frac{1}{12\pi^2}\int_{S^4_{\infty}}|W_{g_{\infty}}^-|^2d\mu_{g_{\infty}} \Big{)} \geq \frac{9}{8}, 
$$ 
and hence it leads a contradiction. 
Therefore, $\mathcal{S} = \emptyset$.  
\end{proof} 

As mentioned in Remark of $\S$\,1, we describe some details on the topological argument. 

\begin{rmk}\label{Counter} 
Let $N_2$ be the nonorientable disk bundle over the the real projective plane $\mathbb{RP}^2$ with Euler number $2$. 
Let $T_4$ be the disk bundle of the complex line bundle over $S^2$ of degree $4$. 
Then, the natural double cover $T_4 \rightarrow N_2$ is the universal cover of $N_2$. 
Note that $S^4 = N_2 \cup_{\partial N_2}N_2$ and $\partial N_2 = S^3/Q_8$ (see \cite{Lawson} for details). 
Note also that $N_2$ is orientable since $N_2$ can be smoothly embedded in $S^4$ as a compact $4$-submanifold. 
Moreover, we have the following: 
$$ 
\ H_1(N_2; \mathbb{Z}) = \mathbb{Z}_2,\quad H_i(N_2; \mathbb{Z}) = 0\ \ (i = 2, 3, 4), 
$$ 
$$ 
H_2(T_4; \mathbb{Z}) = \mathbb{Z},\quad \ \ H_i(T_4; \mathbb{Z}) = 0\ \  (i = 1, 3, 4). 
$$
We do not know whether the orientable open $4$-manifold ${\rm Int}(N_2)$ admits a Ricci-flat ALE metric or not. 
(We have proved here only that such a metric never appears as a bubbling off Ricci-flat ALE metric 
from a sequence in $\mathscr{E}_{\geq c}(S^4)$.) 
However, $N_2$ becomes an orientable counterexample to the topological arguments in the proofs of \cite[Lemma\,6.3]{Anderson} and \cite[Proposition\,3.10]{Anderson-GAFA}.  
\end{rmk} 

\vspace{10mm} 

\bibliographystyle{amsbook}

\vspace{10mm} 

\end{document}